\documentclass[10 pt, twoside]{amsart}
\usepackage{amssymb,latexsym,amsthm,amsmath,mathtools,amsfonts,hyperref,fancyhdr,comment,enumitem, cleveref,xcolor,bbm, caption, amssymb, mathrsfs, bm}
\usepackage[enableskew]{youngtab}
\usepackage[centertableaux]{ytableau}

\hypersetup{
	colorlinks=true,
	linkcolor=blue,
	filecolor=magenta,
	urlcolor=cyan,
	citecolor=red,
}

\long\def\symbolfootnote[#1]#2{\begingroup%
	\def\thefootnote{\fnsymbol{footnote}}\footnote[#1]{#2}\endgroup}

\newcommand{\C}{\ensuremath{\mathscr{C}}}
\newcommand{\p}{\mathscr{P}}
\newcommand{\Z}{\ensuremath{\mathbb{Z}}}

\newcommand{\h}{\textup{ht}}

\newcommand{\tensor}{\otimes}
\newcommand{\Mod}{\mathrm{mod}}
\newcommand{\supp}{\mathrm{Supp}}
\newcommand{\ind}{\mathrm{Ind}}

\newcommand{\irr}{\mathrm{Irr}}
\newcommand{\bst}{\mathrm{BST}}

\makeatletter
\def\imod#1{\allowbreak\mkern10mu({\operator@font mod}\,\,#1)}
\makeatother

\newtheorem{theorem}{Theorem}[section]

\newtheorem*{theorem*}{Theorem}
\newtheorem{definition}[theorem]{Definition}
\newtheorem{observation}[theorem]{Observation}
\newtheorem{remark}[theorem]{Remark}

\newtheorem{example}[theorem]{Example}

\numberwithin{equation}{section}
\newcommand{\ignore}[1]{}

\newcommand{\mynote}[1]{}

\begin{document}

\title[Character Values Of Wreath Products And The Symmetric Group]{A Relationship Between Character Values Of Wreath Products And The Symmetric Group}

\author{Rijubrata Kundu}
\address{The Institute of Mathematical Sciences, Chennai 600113 \\ Homi Bhabha National Institute, Mumbai}
\email{rijubrata8@gmail.com}
	
\author{Papi Ray}
\address{Indian Institute of Technology, Kanpur-208016, India}
\email{popiroy93@gmail.com}
	
\date{\today}

\begin{abstract}
    \noindent A relation between certain irreducible character values of the hyperoctahedral group $B_n$ ($\Z/2\Z \wr S_n$) and the symmetric group $S_{2n}$ was proved by F. L\"ubeck and D. Prasad in 2021.  Their proof is algebraic in nature and uses Lie theory. Using combinatorial methods, R. Adin and Y. Roichman proved a similar relation between certain character values of $G\wr S_n$ and $S_{rn}$, where $G$ is an abelian group of order $r$ (generalizing the result of L\"ubeck-Prasad). Using their result, we prove yet another relation between certain irreducible character values of $G\wr S_n$ and $S_{rn}$, where $G$ is an abelian group of order $r$. 
\end{abstract}

\maketitle

\section{Introduction}\label{intro}

The representation theory of the wreath product $G\wr S_n$, where $G$ is a finite group, was developed by Specht in 1932. While the irreducible characters of the symmetric group $S_n$ are indexed by partitions of $n$, the irreducible characters of $G\wr S_n$ are parameterized by partition-valued functions on the set of irreducible characters of $G$ (see \Cref{sec.pre}). Equivalently, the irreducible characters of $G\wr S_n$ are indexed by $r$-partite partitions of $n$ (which are $r$-tuples of partitions indexed by irreducible characters of $G$), where $r$ is the number of irreducible characters of $G$ . Let $\psi_{\lambda}$ denote the irreducible character of $G\wr S_n$ indexed by an $r$-partite  partition $\lambda$ of $n$, and $\chi_{\nu}$ denote the irreducible character of $S_n$ indexed by a partition $\nu$ of $n$.

\medskip

In \cite{lp}, F. L\"ubeck and D. Prasad proved an interesting relationship between certain irreducible character values of the Weyl group $B_n$ (that is, $\Z/2\Z \wr S_n$) and $S_{2n}$ (which is also the Weyl group $A_{2n-1}$). They proved the following (see \cite[Theorem 1.1]{lp}): For a $2$-partite partition $\lambda$ and a partition $\mu=(\mu_1,\ldots,\mu_t)$ of $n$, $\psi_{\lambda}(1,1,\cdots,1,w_{\mu})=\epsilon(\lambda)\chi_{\hat{\lambda}}(w_{2\mu})$, where $2\mu:=(2\mu_1,\ldots,2\mu_t)$, $w_{\mu}$ (resp. $w_{2\mu}$) denotes the standard representative of the conjugacy class of $S_n$ (resp. $S_{2n}$) indexed by $\mu$ (resp. $2\mu$), $\hat{\lambda}$ is that partition of $2n$ whose $2$-quotient is $\lambda$, and $\epsilon(\hat{\lambda})\in \{\pm 1\}$. The proof of this result is algebraic in nature and uses Lie theory. R. Adin and Y. Roichman reproved this result using combinatorial methods. In fact, they provide a more general statement where $\Z/2\Z$ is replaced by an abelian group $G$ (see \cite[Theorem 4.2]{rr}). Moreover, their statement can be stated in the more general context of any wreath product $G\wr S_n$, and for those irreducible characters of $G\wr S_n$ which \emph{``arise''} from the irreducible characters of $G$ of the same degree. This is stated as \Cref{Theorem_of_RR_general} and its proof follows from a careful analysis of the proof of \cite[Theorem 4.2]{rr}. 

\medskip

Motivated by these results, we state and prove another character relationship between irreducible characters of $G\wr S_n$ evaluated on elements of the form $(a,a,\cdots, a, \pi)$ (where $a\in G$ and $\pi\in S_n$) and certain character values of the symmetric group $S_{rn}$, where $G$ is an abelian group of order $r$. This is stated as \Cref{main_theorem}. Once again, it turns out that our result can be stated more generally for those irreducible characters of $G\wr S_n$ which \emph{``arise''} from the linear characters of $G$, where $G$ is assumed to be any finite group. This is stated as \Cref{main_theorem2}.

\medskip

The paper is organized as follows: In \Cref{sec.pre}, to make the article self-contained, we provide the basic background on the conjugacy classes and irreducible characters of wreath products, and especially the Murnaghan-Nakayama rule which is the most important tool used in this article. In \Cref{sec.exp.rr}, we discuss a mild generalization of \cite[Theorem 4.2]{rr} (\Cref{Theorem_of_RR_general}). Finally, in \Cref{sec.main}, we state and prove our main results (\Cref{main_theorem} and \Cref{main_theorem2}) and include some examples to improve their clarity.

\section{Preliminaries}\label{sec.pre}

We briefly discuss the conjugacy classes and complex irreducible representations of the group $G\wr S_n$ where $G$ is any finite group, and we also set the notations for later use. We follow the expositions in \cite[Appendix B]{ma} and \cite[Chapter 4]{jk}.
\subsection{Conjugacy classes of $G \wr S_n$}

Let $\lambda=(\lambda_1,\lambda_2,\cdots, \lambda_l)$, where $\lambda_1\geq \cdots \geq \lambda_l$ and $\lambda_i\in \mathbb{N}$. We say $\lambda$ is a partition of $|\lambda|=n$ (written $\lambda\vdash |\lambda|=n$), where $|\lambda|=\sum_i\lambda_i=n$ is the sum of its parts. The number of parts of $\lambda$ is denoted by $l(\lambda)$. Alternatively, for a partition $\lambda$ of $n$, we write $\lambda=\langle 1^{m_1},\cdots, i^{m_i},\cdots\rangle$, where $m_i$ is the number of times $i$ occurs as a part in $\lambda$. Let $\p$ denote the set of all partitions (including the empty partition of 0, which will be denoted by $\emptyset$). For $\pi\in S_n$, let $m(\pi)=\langle 1^{m_1},2^{m_2},\cdots \rangle$ denote the cycle-type of $\pi$. Here, $m_i$ denotes the number of $i$-cycles in the disjoint cycle decomposition of $\pi$, whence we have $\sum_i im_i=n$. Thus, $m(\pi)$ yields a partition of $n$. Let $X$ be any finite set. We denote by $\mathscr{P}(X)$ the set of all partition-valued functions on $X$. If $\rho \in \p(X)$, let $\displaystyle ||\rho||=\sum_{c\in X}|\rho(c)|$. Let $\p_n(X)$ denote the set of all partition-valued functions $\rho$ on $X$ such that $||\rho||=n$. 

\medskip

Let $\C=\{C_0,C_1,\cdots,C_{r-1}\}$ be a labelling of the conjugacy classes of $G$. Let $x=(g_1,\cdots,g_n,\pi)\in G\wr S_n$ and $\pi=\pi_1\cdots\pi_k$ be a disjoint cycle decomposition of $\pi$. If $\pi_i=(i_1,i_2,\ldots,i_t)$, then $g_i(x)=g_{i_t}g_{i_{t-1}}\cdots g_{i_1}\in G$ is uniquely determined up to conjugation in $G$. We say that $g_i(x)$ is a cycle product of $g$ corresponding to the cycle $\pi_i$. This allows us to naturally associate $x\in G\wr S_n$ with a partition-valued function $\rho$ on $\C$ as follows: $C_j$ maps to a partition $\lambda=\langle 1^{m_1},\cdots,i^{m_i}, \cdots \rangle$ if and only if there exists $m_i$ cycles of length $i$ in the disjoint cycle decomposition of $\pi$ whose corresponding cycle products belong to $C_j$. It is trivially assumed that if there is no cycle in the disjoint cycle decomposition of $\pi$ whose cycle product lies in $C_j$, then $C_j$ maps to the empty partition of $0$. Since $\pi \in S_n$, it is clear that $||\rho||=n$. We call $\rho$ the type of $x$ and denote it by $\text{ty}(x)$. It is then a standard result that if $x,y\in G\wr S_n$, then $x$ and $y$ are conjugate in $G\wr S_n$ if and only if $\text{ty}(x)=\text{ty}(y)$. Of course, it is easy to see that if $\rho \in \p_n(\C)$, then there exists $x\in G\wr S_n$ such that $\text{ty}(x)=\rho$. Thus, the conjugacy classes of $G\wr S_n$  are parametrized by partition-valued functions $\rho$ on $\C$ such that $||\rho||=n$. If $G$ is trivial, then $\C$ is a singleton set and hence $\p_n(\C)$ is simply the set of all partitions of $n$. Indeed, as is well known, the conjugacy classes of $S_n$ are indexed by partitions of $n$.

\subsection{Complex irreducible characters of $G\wr S_n$} We briefly describe the irreducible characters of $G\wr S_n$. We closely follow the exposition in \cite[Chapter 4]{jk}. Let $\irr(G)$ denote the set of all complex irreducible characters of $G$ and we fix a labelling $\{\chi_{0},\cdots, \chi_{r-1}\}$ of $\irr(G)$. The irreducible characters of $S_n$ are indexed by partitions of $n$. Let $\chi_{\lambda}$ be the irreducible character of $S_n$ indexed by the partition $\lambda\vdash n$ (see \cite{ap,jk}). Let $\phi$ be an irreducible representation of $G$ with representation space $V$ and character $\chi$. Then, $\phi$ gives rise to an irreducible representation of $G\wr S_n$ with representation space $T^n(V)=\underbrace{V\tensor V\tensor \cdots \tensor V}_{n\;\text{times}}$ with the following action: If $x=(g_1,\cdots,g_n,\pi)\in G\wr S_n$, then $x.(v_1\tensor v_2 \tensor\cdots\tensor v_n):=\phi(g_1)(v_{\pi^{-1}(1)})\tensor \cdots \tensor \phi(g_n)(v_{\pi^{-1}(n)})$. We denote this representation by $\widetilde{\tensor^n\phi}$ and its character by $\widetilde{\tensor^n\chi}$. On the other hand, if the irreducible character $\chi_{\lambda}$ of $S_n$ is afforded by the representation $\rho_{\lambda}$ with representation space $V_{\lambda}$, then it trivially lifts to a representation of $G\wr S_n$ with representation space $V_{\lambda}$ by the action $(g_1,\cdots,g_n,\pi).v:=\rho_{\lambda}(\pi)(v)$, for every $v\in V_{\lambda}$. By abuse of notation, we call this representation of $G\wr S_n$ as $\rho_{\lambda}$ and its character as $\chi_{\lambda}$. The internal tensor product of the above two representations of $G\wr S_n$, that is, $\widetilde{\tensor^n \phi} \tensor \rho_{\lambda}$  yields an irreducible representation of $G\wr S_n$ with representation space being $T^n(V)\tensor V_{\lambda}$, and character $\widetilde{\tensor^n \chi} \tensor \chi_{\lambda}$. More generally, if $\rho \in \p_n(\irr(G))$ (with the fixed labelling of $\irr(G)$), and $\rho(\chi_i)=\lambda^i\vdash n_i$ for $0\leq i\leq r-1$, then an irreducible representation (or character) of $G\wr S_n$ can be constructed the following way: for $0\leq i\leq r-1$, let $\phi_i$ be the irreducible representation of $G$ that gives $\chi_i$. Then, $\widetilde{\tensor^n \phi_i} \tensor \rho_{\lambda^i}$ is an irreducible representation of $G\wr S_{n_i}$. The external tensor products of these representations, that is, $\overset{i}{\#}(\widetilde{\tensor^n \phi_i} \tensor \rho_{\lambda^i})$ yields an irreducible representation of $\times_iG\wr S_{n_i}:=G\wr S_{n_1}\times \cdots \times G\wr S_{n_{r-1}}$. Finally, the induction of this representation to $G\wr S_n$ (recall that $\displaystyle \sum_{i=0}^{r-1}n_i=n$), that is, $\ind_{\times_{i}G\wr S_{n_i}}^{G\wr S_n} \#^i(\widetilde{\tensor^n \phi_i} \tensor \rho_{\lambda^i})$ yields a representation of $G\wr S_n$. The character of this representation is $\ind_{\times_{i}G\wr S_{n_i}}^{G\wr S_n} \overset{i}{\#}(\widetilde{\tensor^n \chi_i} \tensor \chi_{\lambda^i})$. These turn out to be inequivalent irreducible representations (in other words, corresponding irreducible characters are distinct) of $G\wr S_n$ corresponding to distinct elements of $\p_n(\irr(G))$. Thus, irreducible characters of $G\wr S_n$ are parametrized by partition-valued functions on the set $\irr(G)$. 

\medskip

For $r\geq 1$, an $r$-partite partition $\lambda$ is an $r$-tuple $(\lambda^0,\cdots,\lambda^{r-1})$, where $\lambda^i\in \p$ for every $0\leq i\leq r-1$. We say $\lambda$ is an r-partite partition of $n$ if $\displaystyle \sum_{i=0}^{r-1}|\lambda^i|=n$. The set of all $r$-partite partitions of $n$ is denoted by $r$-Par(n). Any partition-valued function $\rho$ on $\irr(G)$ (with the fixed labelling) can be naturally identified with an $r$-partite partition $\lambda=(\lambda^0,\cdots,\lambda^{r-1})$, where $\lambda^i=\rho(\chi_i)$. Thus, irreducible characters of $G\wr S_n$ are indexed by $r$-partite partitions of $n$ (with a fixed labelling of the irreducible characters of $G$). Given an r-partite partition $\lambda$ of $n$, we write $\psi_{\lambda}$ to denote the irreducible character of $G\wr S_n$ indexed by $\lambda$.

\subsection{Murnaghan-Nakayama rule for irreducible characters of wreath products} The famous Murnaghan-Nakayama rule is a (recursive) formula for computing the values of irreducible characters of the symmetric group $S_n$ (see \cite[Theorem 5.11.1]{ap}, \cite[Theorem 7.17.3]{stan}). Stembridge proved a Murnaghan-Nakayama type formula for computing the values of the irreducible characters of $G\wr S_n$. We describe this rule here as it is the main tool used in this article. To begin with, we set-up some notations. Let $\lambda\vdash n$ and $T_{\lambda}$ denote the Young diagram of $\lambda$. A border-strip (or a rim-hook) is a connected skew diagram possessing no $2\times 2$ square.  Let $\mu=(\mu_1,\ldots,\mu_t)$ be a composition of $n$, that is, $\mu_i>0$ for all $1\leq i\leq t$, and $\sum_{i}\mu_i=n$.  A border-strip tableaux (abbrv. BST) of shape $\lambda$ and weight $\mu$ is a filling of the young diagram $T_{\lambda}$ such that  (a) $i$ appears $\mu_i$ times, (b) the entries are weakly increasing in both rows and columns, and (c) the skew diagram corresponding to the entries $i$ is a border-strip for every $1\leq i\leq t$. A border-strip tableaux of shape $\lambda$ and weight $\mu$ is equivalent to the successive removal of rim-hooks of size $\mu_i$, starting from $\mu_t$ to $\mu_1$ from $T_{\lambda}$, finally to end up with the empty diagram. Following \cite{rr}, we call this process, a $\mu$-peeling of $\lambda$. The height of a skew-diagram is one less than the number of rows present in it. For any tableaux $T$, let $\h_T(i)$ denote the height of the skew-diagram corresponding to the entry $i$. For convenience, we define $\displaystyle \h(T):=\sum_i \h_T(i)$. These notions can be generalized to $r$-partite partitions in an obvious way. Let  $\lambda=(\lambda^0, \cdots, \lambda^{r-1})$ be an $r$-partite partition, and $T_{\lambda}=(T_{\lambda^0}, \cdots, T_{\lambda^{r-1}})$ be the Young diagram of $\lambda$. A border strip tableaux of shape $\lambda$ and weight $\mu$ is defined to be a filling of $T_{\lambda}$ such that (a) $i$ appears $\mu_i$ times and all these appear in exactly one $T_{\lambda^j}$, (b) the entries are weakly increasing in both rows and columns, and (c) the skew diagram corresponding to the entries $i$ is a border-strip for every $1\leq i\leq t$. Once again, a border strip tableaux of shape $\lambda$ and weight $\mu$ is equivalent to successive removal of rim-hooks of size $\mu_t$ to  $\mu_i$ from $T_{\lambda}$ (that is, for each $1\leq i\leq t$, rim-hook of size $\mu_i$ is removed from some $T_{\lambda^j}$) until we reach the empty $r$-partite partition. For a  border-strip tableaux $T$ of shape $\lambda$ (an $r$-partite partition) and weight $\mu$, we define $f_T(i)$ (called the index of $i$ in $T$) to be $j\in \{0,1,\ldots,r-1\}$ where $i$ appears in $T_{\lambda^j}$ (note that $i$ appears in an unique $T_{\lambda^j}$ by (a)). The following example illustrates all these notions.

\begin{example}
    Let  $\lambda=((2,1),(1,1,1),(2,2),(2))$, and $\mu=(2,3,3,2,2)$. We take a border-strip tableaux of shape $\lambda$ and weight $\mu$ as follows: $\displaystyle T=(\hspace{.2cm}\ytableausetup{smalltableaux}\begin{ytableau} 3&3 \\3 \end{ytableau}\;,\;\begin{ytableau} 2 \\2\\2 \end{ytableau}\;, \;\begin{ytableau} 1&4 \\1&4 \end{ytableau}\;,\;\begin{ytableau} 5&5\end{ytableau}\;).$
    Since $\lambda^{(0)}$ is filled by $3$, so $f_T(3)=0$. For the same reason, $f_T(2)=1$, $f_T(1)=f_T(4)=2$, and $f_T(5)=3$. Further, $\h_T(1)=\h_T(3)=\h_T(4)=1$, $\h_T(2)=2$, and $\h_T(5)=0$. Thus, $\displaystyle \h(T)=\sum_i \h_T(i)=5$.
\end{example}

Now we are ready to state Stembridge's Murnaghan-Nakayama rule for the irreducible characters of $G\wr S_n$.

\begin{theorem}\cite[Theorem 4.3]{st}\label{MN_rule_wreath_product}
    Let $G$ be a finite group and $\irr(G)=\{\chi_0,\cdots,\chi_{r-1}\}$ be a labelling of the irreducible characters of $G$. Let $\psi_{\lambda}$ denote the irreducible character of $G\wr S_n$ indexed by an $r$-partite partition $\lambda$ (with respect to the labelling of $\irr(G)$). Let $x=(g_1,\cdots,g_n,\pi)\in G\wr S_n$, and  $\pi=\pi_1\cdots\pi_t$ be a disjoint cycle decomposition of $\pi$ with $l(\pi_i)=\mu_i$ for every $1\leq i\leq t$. Here, $l(\pi_i)$ denotes the length of the cycle $\pi_i$. Then,

    $$\psi_{\lambda}(x)=\sum_{T\in \mathrm{BST}(\lambda,\mu)}\prod_{i=1}^{t} (-1)^{\h_T(i)}\chi_{f_T(i)}(g_i(x)),$$
    where $\mu=(\mu_1,\cdots,\mu_t)$  and $\mathrm{BST}(\lambda,\mu)$ is the set of all border-strip tableaux of shape $\lambda$ and weight $\mu$.
\end{theorem}

We note that if $G$ is trivial, we get back the classical Murnaghan-Nakayama rule for irreducible characters of $S_n$ (see \cite[Theorem 7.15]{stan}), that is, for $\lambda\vdash n$ and $\pi\in S_n$
\begin{equation}\label{MN_rule_Sym_group}
    \chi_{\lambda}(\pi)=\sum_{T\in \mathrm{BST}(\lambda,\mu)}(-1)^{\h(T)},
\end{equation}

where $\pi=\pi_1\cdots\pi_t$ is a disjoint cycle decomposition of $\pi$ and $l(\pi_i)=\mu_i$ as before, which yields the weak composition $\mu=(\mu_1,\cdots,\mu_t)$ of $n$.

\section{The Theorem of Roichman-Adin}\label{sec.exp.rr}

In this section, we briefly discuss the result of Roichman-Adin and a mild generalisation that easily follows from the proof of their theorem. Their result involves the well-known notions of cores and quotients of a partition (see \cite[2.7]{jk}). Let $\lambda\vdash n$ and $r\in \mathbb{N}$. Let $\tilde{\lambda}$ denote the $r$-core of $\lambda$ and $[\lambda]_r$ denote the $r$-quotient of $\lambda$. By definition, $[\lambda]_r$ is an $r$-partite partition of $k$, where $k$ is the number of rim-hooks of length $r$ removed from $\lambda$ to obtain its $r$-core $\tilde{\lambda}$. It is a well-known theorem that the pair $(\tilde{\lambda}, [\lambda]_r)$ determines $\lambda$ uniquely. All these can be proved and visualized using the notion of an abacus diagram (see \cite[2.7]{jk} for details). For $r,n\in \mathbb{N}$, let $\lambda$ be a partition of $rn$ with an empty $r$-core, that is, $\tilde{\lambda}=\emptyset$. The $r$-quotient $[\lambda]_r$ is then an $r$-partite partition of $n$. Let $\p_{rn}(\emptyset)$ denote the set of all partitions of $rn$ with an empty $r$-core. Then, $\lambda \mapsto [\lambda]_r$ is a bijection between the sets $\p_{rn}(\emptyset)$ and $r$-Par($n$). Let us denote the converse map by $\lambda\mapsto \hat{\lambda}$, where $\lambda$ is an $r$-partite partition of $n$ and $\hat{\lambda}$ is a partition of $rn$ with an empty $r$-core such that its $r$-quotient is $\lambda$. In \cite[Section 2]{rr}, the authors describe all these in detail using $0$-$1$ encoding of partitions. The construction of $\hat{\lambda}$ from $\lambda$ (using $0$-$1$ encoding) is of particular importance as it is used in the proof of their main theorem, which is stated below.

\begin{theorem}\cite[Theorem 4.2]{rr}\label{Theorem_of_RR}
Let $G$ be an abelian group of order $r$ and $\lambda$ be an $r$-partite partition. Let $x=(e,\ldots, e,\pi)\in G\wr S_n$ (where $e$ is the identity of $G$) and $\pi=\pi_1\cdots\pi_t$ be a disjoint cycle decomposition of $\pi$, with $l(\pi_i)=\mu_i$ for every $1\leq i\leq t$.  Then,
    $$\psi_{\lambda}(x)=\mathrm{sign}_{r}(\hat{\lambda})\chi_{\hat{\lambda}}(w_{r\mu}),$$
    where $w_{r\mu}\in S_{rn}$ has cycle-type given by $r\mu:=(r\mu_1,\cdots,r\mu_t)$, and $\mathrm{sign}_r:\p_{rn}(\emptyset)\to \{\pm 1\}$ is the map in \cite[Definition 3.7]{rr}.
\end{theorem}

In order to describe a mild generalization of the above result, we summarize the main steps of the proof of the above result. With the assumption of the theorem and using \Cref{MN_rule_wreath_product}, it follows that
\begin{equation}\label{eqn_1}
    \psi_{\lambda}(x)=\sum_{T\in \mathrm{BST}(\lambda,\mu)} (-1)^{\h(T)}
\end{equation}

One of the main step is to realize the above sum in terms of $\hat{\lambda}$. Using the  interpretation of cores and quotients using the $0$-$1$ encoding, the authors (see \cite[Theorem 2.11]{rr}) prove that

\begin{equation}\label{eqn_2}
    \sum_{T\in \mathrm{BST}(\lambda,\mu)} (-1)^{\h(T)}=\sum_{T\in \mathrm{BST}(\hat{\lambda},r\mu)} (-1)^{\h_0(T)},
\end{equation}
where $r\mu:=(r\mu_1,\cdots, r\mu_t)$ and $\h_0(T)$ is a statistic corresponding to $T$ (similar to the usual $\h(T)$ we have defined) that can be interpreted using the $0$-$1$ encoding of $\hat{\lambda}$ (step 4, main loop, \cite[Theorem 2.11]{rr}).  Finally, the authors show that given a border-strip tableaux $T$ of shape $\hat{\lambda}$ and weight $r\mu$, the two statistic $(-1)^{\h(T)}$ and $(-1)^{\h_0(T)}$ differ by a sign, which is irrespective of the tableaux $T$. This sign is determined by the map $\mathrm{sign}_r:\p_{rn}(\emptyset)\to \{\pm 1\}$, that is, 

\begin{equation}\label{eqn_3}
    \sum_{T\in \mathrm{BST}(\hat{\lambda},r\mu)} (-1)^{\h_0(T)}=\mathrm{sign}_r(\hat{\lambda})\sum_{T\in \mathrm{BST}(\hat{\lambda},r\mu)} (-1)^{\h(T)},
\end{equation}

\noindent where $\mathrm{sign}_r(\hat{\lambda})$ depends only on $r$ and $\hat{\lambda}$. By the classical Murnaghan-Nakayama rule for the symmetric groups, RHS of the above equation is clearly $\mathrm{sign}_r(\hat{\lambda})\chi_{\hat{\lambda}}(w_{r\mu})$ and hence the result.

\begin{remark}
    Although the definition of the map $\mathrm{sign}_r:\p_{rn}(\emptyset)\to \{\pm 1\}$ appears a little bit technical (using $0$-$1$ sequences), a close look at its definition provides an alternative description for it in terms of character values of $S_{rn}$. Let $\eta=\underbrace{(r,r,\cdots,r)}_{\text{$n$ times}}\vdash rn$, $w_{\eta}$ be a disjoint product of $r$ many $n$-cycles, and $\lambda\in \p_{rn}(\emptyset)$. For any $T\in \bst(\lambda, \eta)$ (which is equivalent to a sequence of removal of $n$ many rim-hooks of length $r$), it can be shown that $(-1)^{\h(T)}$ is independent of $T$, by using the notion of numbered bead configuration of an abacus diagram (\cite[2.7.25]{jk}). It turns out that  $(-1)^{\h(T)}$ is precisely $\mathrm{sign}_r(\lambda)$, that is, $\chi_{\lambda}(w_{\eta})=\mathrm{sign}_r(\lambda)|\bst(\lambda, \eta)|$ (using \Cref{MN_rule_Sym_group}). In other words, $\mathrm{sign}_r(\lambda)=\frac{\chi_{\lambda}(w_{\eta})}{|\chi_{\lambda}(w_{\eta})|}$. \cite[2.7.32]{jk} yields a formula for $|\bst(\lambda, \eta)|$ which is clearly seen to equal the degree of the character $\psi_{[\lambda]_r}$. Of course, this is also a special case of \Cref{Theorem_of_RR} as is also mentioned in \cite{lp} for $r=2$. In fact, $\mathrm{sign}_2(\lambda)=(-1)^{\mathrm{odd}(\lambda)}$, where $\mathrm{odd}(\lambda)$ is the number of odd parts in $\lambda$ (see \cite[Proposition 4.2]{lp}). We reiterate \textbf{Question 5.5} of \cite{rr} that asks for a similar simple description of $\mathrm{sign}_r(\lambda)$ for $r>2$.
\end{remark}

To state our mild generalisation, we need a definition. For an $r$-partite partition $\lambda=(\lambda^0,\cdots,\lambda^{r-1})$, we define $\bm{\supp(\lambda):=\{i \mid \lambda^i\neq \emptyset\}}$. Recall that $\psi_{\lambda}$ is the irreducible character of $S_n$ indexed by $\lambda$ with respect to the ordering $\irr(G)=\{\chi_0,\cdots,\chi_{r-1}\}$. We can now state the generalization.

\begin{theorem}\label{Theorem_of_RR_general}
    Let $G$ be a finite group of order $r$ and $\lambda$ be an $r$-partite partition. Assume that there exists $d\in \mathbb{N}$ such that $\deg(\chi_i)=d$, for every $i\in \supp(\lambda)$. Let $x=(e,\ldots, e,\pi)\in G\wr S_n$ (where $e$ is the identity of $G$) and $\pi=\pi_1\cdots\pi_t$ be a disjoint cycle decomposition of $\pi$, with $l(\pi_i)=\mu_i$ for every $1\leq i\leq t$.  Then,
    $$\psi_{\lambda}(x)=d^t\mathrm{sign}_{r}(\hat{\lambda})\chi_{\hat{\lambda}}(w_{r\mu}),$$
    where $w_{r\mu}\in S_{rn}$ has cycle-type given by $r\mu:=(r\mu_1,\cdots,r\mu_t)$, and $\mathrm{sign}_r:\p_{rn}(\emptyset)\to \{\pm 1\}$ is the map in \cite[Definition 3.7]{rr}.
\end{theorem}

\begin{proof}
    For $T\in \mathrm{BST}(\lambda,\mu)$ and $1\leq i\leq t$, note that $f_T(i)\in \supp(\lambda)$ and hence $\chi_{f_T(i)}(e)=d$. Thus, using \Cref{MN_rule_wreath_product} we get that 
    $$\displaystyle \psi_{\lambda}(x)=d^t\sum_{T\in \mathrm{BST}(\lambda,\mu)}(-1)^{\h(T)}.$$
    Now the result follows from \Cref{eqn_2} and \Cref{eqn_3}.
\end{proof}

\section{The main theorem}\label{sec.main}
In this section, we will state and prove the main theorem, namely Theorem \ref{main_theorem} of this article. Some examples are also given to illustrate the theorem. To state the result, we need a few notations. For $r\geq 2$, let \bm{$\zeta=e^{2\pi i/r}$} and $C_r=\langle x \rangle$ be the cyclic group of order $r$ generated by an element $x$. Set $\irr(C_r)=\{\theta_0,\ldots, \theta_{r-1}\}$, where $\theta_i$ is the irreducible character  of $C_r$ defined by $a\mapsto \zeta^i$.
\begin{theorem}\label{main_theorem}
    Let $G$ be a finite abelian group of order $d$, and $\irr(G)=\{\chi_0,\chi_1,\cdots,\chi_{d-1}\}$ denote a fixed labelling of the irreducible characters of $G$. Let $a\in G$ be an element of order $r$ and $m=d/r$. Assume that for all $0\leq k\leq r-1$, $\{\chi_{1^{(k)}}, \chi_{2^{(k)}}, \cdots, \chi_{m^{(k)}} \}$ be the set of all those irreducible characters of $G$ (with respect to the labelling of $\irr(G)$) whose restriction to $H:=\langle a \rangle$ is $\theta_k$. Let $\lambda=(\lambda^0,\cdots,\lambda^{d-1})$ be a $d$-partite partition of $n$ with $\lambda^i\vdash n_i$ for every $0\leq i\leq d-1$, and $\psi_{\lambda}$ denote the irreducible character of $G\wr S_n$ indexed by $\lambda$ (w.r.t the fixed labelling of $\irr(G)$). Then,
    $$\psi_{\lambda}(a,a,\ldots,a,\pi)=\zeta^{\alpha}\mathrm{sign}_d(\hat{\lambda})\chi_{\hat{\lambda}}(w_{d\mu}),$$
    where $\displaystyle \alpha=\sum_{k=0}^{r-1}\sum_{i=1}^{m}kn_{i^{(k)}}$, $\pi\in S_n$, $\pi=\pi_1\pi_2\cdots\pi_t$ be a disjoint cycle decomposition of $\pi$ with $l(\pi_i)=\mu_i$ for every $1\leq i\leq t$ (whence $\mu:=(\mu_1,\cdots, \mu_t))$, and $d\mu:=(d\mu_1,d\mu_2,\cdots,d\mu_t)$.
\end{theorem}

\begin{proof}
    We first assume that $G$ is a cyclic group of order $d$ and $a\in G$ is of order $d$, that is, $a$ generates $G$. In this case $m=1$, and hence for every $0\leq k\leq d-1$, $\chi_{1^{(k)}}=\theta_k$ (where $\theta_k:a\mapsto \zeta^k$). Set $y=(a,\ldots,a,\pi)$. For $T\in \mathrm{BST}(\lambda,\mu)$, let $\displaystyle R_T=\prod_{i=1}^{t}\chi_{f_T(i)}(g_i(y))$. Notice that $g_i(y)=a^{\mu_i}$, whence if $\mu_i\equiv j(\Mod\; r)$ where $0\leq j\leq r-1$, then $g_i(y)=a^j$.  Consider the filling in each single ordinate $T_{\lambda^s}$ ($0\leq s\leq d-1$) separately. For each $1\leq i\leq t$, whenever $j\in T_{\lambda^s}$, note that $f_T(i)=s$. Thus, for each such $i$ appearing in $T_{\lambda^s}$,  if $\mu_i=j(\Mod\;r)$, then $\chi_{f_T(i)}(g_i(y))=\chi_{s}(a^j)$. Now, there exists $0\leq k\leq d-1$ such that $\chi_{s}=\chi_{1^{(k)}}=\theta_{k}$, whence $\chi_s(a^j)=\zeta^{kj}$. For $0\leq j\leq r-1$, let us now assume that $T_{\lambda^s}$ consists of  $a_j$ many distinct $i$'s with $\mu_i\equiv j\;(\Mod\;r)$. Note that $a_j=0$ if no such $i$ occurs. Then,
    $$\displaystyle \prod_{\substack{i=1 \\ i\in T_{\lambda^s}}}^{t} \chi_{f_T(i)}(g_i(y))= \zeta^{\sum\limits_{j=0}^{r-1} ka_jj}=\zeta^{k n_s}.$$
    The last equality follows as $\lambda^s\vdash n_s$ and $\displaystyle n_s=\sum_{j=0}^{r-1}\sum_{\substack{i\in T_{\lambda^s} \\ \mu_i\equiv j\;(\Mod\; r)}}\mu_i \equiv \sum_{j=0}^{r-1}ja_j\;(\Mod\;r)$. Finally, by our notation, $\zeta^{kn_s}=\zeta^{kn_{1^{(k)}}}$. We conclude that 
    $$\displaystyle R_T=\prod_{s=0}^{d-1}\prod_{\substack{i=1 \\ i\in T_{\lambda^s}}}^{t} \chi_{f_T(i)}(g_i(y))=\zeta^{\alpha},$$
    where $\displaystyle \alpha=\sum_{k=0}^{d-1}n_{1^{(k)}}$. Thus, $R_T$ is independent of the tableaux $T$. Using \Cref{MN_rule_wreath_product}, we get
    $$\displaystyle \psi_{\lambda}(y)=\sum_{T\in \bst(\lambda,\mu)}(-1)^{\h(T)}R_T=\zeta^{\alpha} \sum_{T\in \bst(\lambda,\mu)}(-1)^{\h(T)}=\zeta^{\alpha}\mathrm{sign}_d(\hat{\lambda})\chi_{\hat{\lambda}}(w_{d\mu}).$$
    The last equality follows from the discussion below \Cref{Theorem_of_RR}. Our result follows in this case. See \Cref{example_1} for an illustration of this case.

    \medskip

    Now, we assume that $G$ is an abelian group of order $d$ and $a\in G$ has order $r$. Let $m=d/r$. If $y=(a,\cdots,a,\pi)$, then $g_i(y)=a^{\mu_i}$ for $1\leq i\leq t$. This means that $\chi_{f_T(i)}(g_i(y))=\chi_{f_T(i)}(a^{\mu_i})=\chi_{f_T(i)}|_{H}(a^{\mu_i})$, where $H=\langle a \rangle$. If $\chi \in \irr(G)$, then $\chi|_{H}\in \irr(H)$. Hence the result in this case follows from the previous case, once we show the validity of the hypotheses that for each $0\leq i\leq r-1$, there are exactly $m$ many irreducible characters of $G$ whose restriction to $H$ is $\theta_i$ ($\theta_i:a \mapsto \zeta$, where $\zeta=e^{2\pi i/r}$). Clearly, if $H$ is any subgroup of $G$, and $\varphi:H\to \mathbb{C}^{\times}$  be a character of $H$, then $\varphi$ can be extended to a character of $G$. This can be easily proved using induction on $[G:H]$ and we omit the details here. Let $X(G)$ denote the set of all multiplicative characters of $G$. $X(G)$ forms a group with $((\chi_1+\chi_2)(g):=\chi_1(g)\chi_2(g), \;\text{where}\;\chi_1,\chi_2\in X(G))$. Define the map $\Gamma:X(G)\to \mathbb{C}^{\times}$ by $\chi\mapsto \chi(a)$. Clearly, this map is a homomorphism with $\mathrm{ker}(\Gamma)=\{\chi\in X(G)\mid  \chi(a)=1\}$. Since any character of $H$($=\langle a\rangle$) can be extended to a character of $G$, it follows that $\Gamma(X(G))=\langle \zeta \rangle$. Since $|X(G)|=|G|=d$ and $|\Gamma(X(G))|=r$, it follows that $|\mathrm{Ker}(\Gamma)|=d/r=m$. More precisely, if $\chi|_H=\theta_i$, then every element of $\chi\mathrm{Ker}(\Gamma)$ restriced to $H$ is $\theta_i$ and we get our desired result. The proof is now complete. See \Cref{example_2} and \Cref{example_3} for illustrations.
\end{proof}

\begin{example}\label{example_1}
    Consider $G=\Z/6\Z=\{\bar{0}, \bar{1},\cdots, \bar{5}\}$.  Let $\chi_j:\Z/6\Z\to \mathbb{C}^{\times}$ be defined by $\bar{1}\mapsto \zeta^j$, where $\zeta=e^{\pi i/3}$. Now, we fix the labelling $\irr(G)=\{\chi_0,\chi_1,\chi_2,\chi_3,\chi_4,\chi_5\}$. Let $\lambda=(\lambda^{0},\cdots,\lambda^5)$ be a 6-partite partition of $n$ such that $\lambda^i\vdash n_i$ for all $1\leq i\leq 5$, and $\psi_{\lambda}$ be the irreducible character of $\Z/6\Z\wr S_n$ indexed by $\lambda$. Let $y=(\bar{5},\cdots,\bar{5},\pi)$, where $\pi\in S_n$ be as in the statement of \Cref{main_theorem}. Now, $\Z/6\Z=\langle \bar{5}\rangle$ and $\theta_i$ is defined by $\bar{5}\mapsto \zeta^i$. We observe that $\chi_0=\theta_0=\chi_{1^{(0)}}$, $\chi_1=\theta_5=\chi_{1^{(5)}}, \chi_2=\theta_4=\chi_{1^{(4)}}, \chi_3=\theta_3=\chi_{1^{(3)}}, \chi_4=\theta_2=\chi_{1^{(2)}}, \chi_5=\theta_1=\chi_{1^{(1)}}$. Thus, by our theorem, $\alpha=5n_1+4n_2+3n_3+2n_2+n_5$ and $\psi_{\lambda}(y)=\zeta^{\alpha}\mathrm{sign}_6(\hat{\lambda})\chi_{\hat{\lambda}}(w_{6\mu})$. 

    If $y=(\bar{1},\cdots,\bar{1},\pi)$, then $\psi_{\lambda}(y)=\zeta^{\alpha}\mathrm{sign}_6(\hat{\lambda})\chi_{\hat{\lambda}}(w_{6\mu})$, where $\alpha=n_1+2n_2+3n_3+4n_4+5n_5$. Recall that by \Cref{Theorem_of_RR}, $\psi_{\lambda}(\bar{0},\cdots,\bar{0},\pi)=\mathrm{sign}_6(\hat{\lambda})\chi_{\hat{\lambda}}(w_{6\mu}).$
\end{example}

\begin{example}\label{example_2}
    In the previous example, consider $y=(\bar{2},\cdots, \bar{2},\pi)$. Then, $H=\langle \bar{2} \rangle$ has order $3$. In this case, $r=3$ and $m=2$. Then, for $0\leq j\leq 2$, $\theta_j: \bar{2}\mapsto \omega^j$, where $\omega=e^{2\pi i/3}$. We fix the same ordering of $\irr(G)$ as in the previous example. Then, (a) $\chi_0|_H=\theta_0=\chi_{1^{(0)}}|_H$ and $\chi_3|_H=\theta_0=\chi_{2^{(0)}}|_H$, (b) $\chi_1|_H=\theta_1=\chi_{1^{(1)}}|_H$ and $\chi_4|_H=\theta_1=\chi_{2^{(1)}}|_H$, and (c) $\chi_2|_H=\theta_2=\chi_{1^{(2)}}|_H$ and $\chi_5|_H=\theta_2=\chi_{2^{(2)}}|_H$. Thus, from \Cref{main_theorem}, we obtain $\alpha=n_1+n_4+2n_2+2n_5$ and $\psi_{\lambda}(y)=\omega^{\alpha}\mathrm{sign}_6(\hat{\lambda})\chi_{\hat{\lambda}}(w_{6\mu})$. 
\end{example}

\begin{example}\label{example_3}
    Let $G=\Z/2\Z \times  Z/2\Z=\{(\bar{0},\bar{0}),(\bar{0},\bar{0}), (\bar{1},\bar{0}),(\bar{0},\bar{1}), (\bar{1},\bar{1}))\}=\langle (\bar{1},\bar{0}),(\bar{0},\bar{1}) \rangle$. For simplicity, let $a:=(\bar{1},\bar{0})$, $b:=(\bar{0},\bar{1})$, and  $c:=ab=(\bar{1},\bar{1})$. We fix a labelling of $\irr(G)=\{\chi_0,\chi_1,\chi_2,\chi_3\}$ where $\chi_0: a,b\mapsto 1$, $\chi_1:a\mapsto -1,b\mapsto 1$, $\chi_2:a\mapsto 1, b\mapsto -1$, $\chi_3:a,b\mapsto -1$. Let $y=(c,\cdots,c,\pi)\in G\wr S_n$. Set $H:=\langle c \rangle$. Then $|H|=2$, $r=2$, and $m=2$. For $0\leq i\leq 1$, we have $\theta_i:c\mapsto (-1)^i$. Further, (a) $\chi_{1^{(0)}}|_H=\theta_0=\chi_0|_H\;\text{and}\;\chi_{2^{(0)}}|_H=\theta_0=\chi_3|_H$, (b) $\chi_{1^{(1)}}|_H=\theta_1=\chi_1|_H\;\text{and}\;\chi_{2^{(1)}}|_H=\theta_1=\chi_2|_H$. Thus, $\alpha=n_1+n_2$ and $\psi_{\lambda}(y)=(-1)^{\alpha}\mathrm{sign}_4(\hat{\lambda})\chi_{\hat{\lambda}}(w_{4\mu})$. 
\end{example}

Similar to \Cref{Theorem_of_RR_general}, we can state \Cref{main_theorem} in the general set-up of any finite group $G$ (not just abelian).  Recall that if $G$ is any finite group, the multiplicative characters (equivalently, irreducible representations of dimension 1) of $G$ are given by those of $G/G'$, where $G'$ is the commutator subgroup of $G$. If $\tau:G\to G/G'$ is the canonical homomorphism and $\chi$ is a multiplicative character of $G$, then $\chi=\varphi \circ \tau$ for some multiplicative character $\varphi$ of $G/G'$. 

\medskip

Let $G$ be a finite group of order $d$ and $[G:G']=s$. Fix a labelling of the irreducible characters, $\irr(G)=\{\chi_0,\chi_1,\cdots,\chi_{d-1}\}$ such that  $\deg(\chi_i)=1$ for every $0\leq i\leq s-1$. Let $\irr(G/G')=\{\varphi_0,\cdots,\varphi_{s-1}\}$ and $\chi_i=\tau\circ\varphi_{i}$ for every $0\leq i\leq s-1$. 

\begin{theorem}\label{main_theorem2}
     With notations as above, Let $a\in G$ be such that $aG'$ has order $r$ in $G/G'$. Set $m=s/r$. Assume that for all $0\leq k\leq r-1$, $\{\varphi_{1^{(k)}}, \varphi_{2^{(k)}}, \cdots, \varphi_{m^{(k)}} \}$ be the set of all those irreducible characters of $G/G'$ (with respect to the labelling of $\irr(G/G')$) whose restriction to $H:=\langle aG' \rangle$ is $\theta_k$. Let $\lambda=(\lambda^0,\cdots,\lambda^{d-1})$ be a $d$-partite partition of $n$ with $\lambda^i\vdash n_i$ for every $0\leq i\leq d-1$, $\supp(\lambda)\subseteq \{0,1,\ldots,s-1\}$ (in other words, $n_i=0$ for every $i\geq s$), and $\psi_{\lambda}$ denote the irreducible character of $G\wr S_n$ indexed by $\lambda$ (w.r.t the fixed labelling of $\irr(G)$). Then,
     $$\psi_{\lambda} 
    (a,a,\ldots,a,\pi)=\zeta^{\alpha}\mathrm{sign}_d(\hat{\lambda})\chi_{\hat{\lambda}}(w_{d\mu}),$$
    where $\displaystyle \alpha=\sum_{k=0}^{r-1}\sum_{i=1}^{m}kn_{i^{(k)}}$, $\pi\in S_n$, $\pi=\pi_1\pi_2\cdots\pi_t$ be a disjoint cycle decomposition of $\pi$ with $l(\pi_i)=\mu_i$ for every $1\leq i\leq t$ (whence $\mu:=(\mu_1,\cdots, \mu_t))$, and $d\mu:=(d\mu_1,d\mu_2,\cdots,d\mu_t)$.
\end{theorem}

\begin{proof}
    The proof is alike the proof of \Cref{Theorem_of_RR}. \Cref{example_4} provides an illustration.
\end{proof}

\begin{example}\label{example_4}
    Let $G=S_3$ and $\irr(G)=\{\chi_{(3)},\chi_{(1^3)},\chi_{(2,1)}\}$ be a fixed labelling of $\irr(G)$. The commutator subgroup of $S_3$ is $A_3$, and $S_3/A_3=\{A_3, \sigma A_3\}$ where $\sigma$ is an odd permutation. Let $\lambda=(\lambda^0,\lambda^1,\emptyset)$ be a 3-partite partition with $\lambda^i\vdash n_i$ for $i=0,1$. Let $y=(\eta,\ldots, \eta, \pi)\in S_3\wr S_n$ and $\pi$ be as in \Cref{main_theorem2}. If $\eta\in A_3$, then $\psi_{\lambda}(y)=\mathrm{sign}_6(\hat{\lambda})\chi_{\hat{\lambda}}(w_{6\mu})$. If $\eta\in S_3\setminus A_3$, then $\psi_{\lambda}(y)=(-1)^{n_2}\mathrm{sign}_6(\hat{\lambda})\chi_{\hat{\lambda}}(w_{6\mu})$.
\end{example}

\subsection*{Acknowledgement}

We thank Prof. Arvind Ayyer for fruitful discussions.
\bibliographystyle{alpha} 
\bibliography{Biblography}

\end{document}